\documentclass[10pt]{amsart}
\usepackage{amscd, amssymb}

\newtheorem{Thm}{Theorem}

\newtheorem{Prop}[Thm]{Proposition}

\newtheorem{Def/Thm}[Thm]{Definition/Theorem}

\newtheorem{Lemma}[Thm]{Lemma}

\theoremstyle{remark}

\numberwithin{equation}{subsection}

\newcommand{\ra }{\rightarrow}

\newcommand{\PP }{{\mathbb P}}

\newcommand{\CC }{{\mathbb C}}

\begin{document}

\title{A Generalization of Fulton-MacPherson Configuration Spaces}

\author{Bumsig Kim}
\address{School of Mathematics, Korea Institute for Advanced Study,
Hoegiro 87, Dongdaemun-gu, Seoul, 130-722, Korea}
\email{bumsig@kias.re.kr}

\author{Fumitoshi Sato}
\address{Takamatsu National College of Technology, 355 Chokushi-cho, Takamatusu-shi, Kagawa,
761-8058, Japan}
\email{fumi@kias.re.kr}

\begin{abstract}
We construct  a wonderful compactification of the variety
parameterizing  $n$ distinct labeled points in $X$ away from $D$,
where $X$ is a nonsingular variety and $D$ is a nonsingular proper
subvariety. When $D$ is empty, it coincides with the
Fulton-MacPherson configuration space.
\end{abstract}


 \maketitle

\section{Introduction}

\subsection{} Let $X$ be a complex connected
nonsingular algebraic variety $X$ and let $D$ be a nonsingular
closed proper subvariety of $X$. The goal of this paper is to
construct the following two spaces:

\begin{itemize}

\item A compactification $X_D^{[ n ]}$ of the configuration space
of $n$ labeled points in $X\setminus D$, \lq\lq not allowing the
points to meet $D$."

\item A compactification $X_{D}[ n ]$ of the configuration space
of $n$ {\em distinct} labeled points in $X\setminus D$, \lq\lq not
allowing the points to
 meet each other as well as $D$."

\end{itemize}

To describe the constructions, we introduce some notation. Let
$D=\bigcup _c D_c$ where $D_c$ are irreducible components of $D$.
For a subset $S$ of $N:=\{1,2,...,n\}$ denote by $D_{c,S}$ the
collection of points $x$ in $X^n$ whose $i$-th
component $x_i$ is in $D_c$ if $i\in S$.
For a subset $I$ (with $|I|\ge 2$) of $N$ let $\Delta _I\subset
X^n$ be the diagonal consisting of $x$ satisfying $x_i=x_j$
whenever $i,j\in I$. We denote by $\mathrm{Bl}_ZX$ the blowup of a
variety $X$ along a closed subvariety $Z$.

Then:

\begin{itemize}

\item Define $X_D^{[n]}$  to be the closure of
$X^n\setminus\bigcup _{c,S}D_{c,S}$ diagonally embedded in
$$X^n\times \prod _{c,\ S \subset N,\ |S|\ge 1} \mathrm{Bl}_{D_{c, S}}X^n .$$

\item Define $X_D[n]$ to be the closure of $(X\setminus
D)^n\setminus \bigcup _{|I|\ge 2} \Delta _I$ in the product
\[ X^{[n]}_{D}\times \prod _{I\subset N,\  |I|\ge 2}
\mathrm{Bl}_{\widetilde{\Delta} _I} X^{[n]}_D , \] where
$\widetilde{\Delta}_I$ is a proper transform of $\Delta_I$.
\end{itemize}

These spaces satisfy wonderful properties as follows.

\begin{Thm}\label{D}

\begin{enumerate}

\item\label{D1} The variety $X_D^{[n]}$ is nonsingular.

\item\label{D2} There is a \lq\lq universal" family $X_D^{[n] +}
\ra X_D^{[n]}$: It is a flat family of stable degenerations of $X$
with $n$ smooth labeled points away from $D$.

\item\label{D3} The boundary $X_D^{[n]}\setminus (X^n\setminus\bigcup _{c,S}D_{c,S})$
is a union of divisors $\widetilde{D}_{c,S}$ corresponding to
$D_{c,S}$, $|S|\ge 1$. Any set of these divisors intersects
transversally.

\item\label{D4} The intersection of boundary divisors
$\widetilde{D}_{c_1,S_1},...,\widetilde{D}_{c_a,S_a}$ is nonempty
if and only if they are nested in the sense that each pair $S_i$
and $S_k$ is:

\begin{itemize}

\item  disjoint if $c_i\ne c_k$;

\item one is contained in the other if $c_i = c_k$.

\end{itemize}

\end{enumerate}
\end{Thm}

\begin{Thm}\label{Delta}

\begin{enumerate}

\item The variety $X_D[n]$ is nonsingular.

\item\label{Delta2} There is a \lq\lq universal" family $X_D[n]^+
\ra X_D[n]$: It is a flat family of stable degenerations of $X$
with $n$ {\em distinct} smooth labeled points away from $D$.

\item The boundary $X_D[n]\setminus ((X\setminus D)^n\setminus
\cup _I \Delta _I)$ is a union of divisors $\widetilde{D}_{c,S}$
and $\widetilde{\Delta }_I$, corresponding to $D_{c,S}$, $|S|\ge
1$, and $\Delta _I$ with $|I|\ge 2$. Any set of these divisors
intersects transversally.

\item\label{nested} The intersection of boundary divisors $
\widetilde{D}_{c_1,S_1},...,\widetilde{D}_{c_a,S_a},\,
\widetilde{\Delta} _{I_1},...,\widetilde{\Delta} _{I_b}$ are
nonempty if and only if they are nested.  Here the collection $\{
\widetilde{D}_{c_i,S_i}, \widetilde{\Delta} _{I_j} \} _{1\le i \le
a, \ 1\le j\le b}$ is called nested if $\{ \widetilde{D}_{c_i,S_i}
\} _{1\le i \le a}$ is nested; for each pair $I_j$ and $I_l$
either they are disjoint or one is contained in the other; and for
each pair $S_i$ and $I_k$, either they are disjoint or $I_k$ is
contained in $S_i$.

\end{enumerate}

\end{Thm}

When $D$ is empty, then the construction of $X_D[n]$ is exactly
the Fulton-MacPherson compactification $X[n]$ of the configuration
space of $n$ distinct labeled points in $X$ (\cite{FM}).  The
meaning of the statements (2) in Theorems will be explained in
subsection \ref{expuniv}. For the definitions of
$\widetilde{D}_{c,S}$ and $\widetilde{\Delta }_I$, see subsection
\ref{notation}.

\medskip

To prove Theorems \ref{D} and \ref{Delta}, we use L. Li's general
work on wonderful compactifications (\cite{DP, MP, Hu, Li1}). For
the history of wonderful compactifications, we refer the reader to
\cite{Li1}. One may show our Theorems also by the conical
wonderful compactification (\cite{MP}).
The Chow rings and motives of the spaces constructed here are described
in \cite{Sa}.

\medskip

Our motivation for the construction of the spaces $X_D^{[n]}$ and $X_D[n]$ is their
use in the study of stable relative maps and stable relative (un)ramified maps, respectively. This will
be studied in detail elsewhere; here we give only a rough explanation of this application.
First note that one can interpret the stable relative maps of \cite{J.Li} as maps from curves to
the fibers of the universal family $X_D^{[n]+}$. Next, the paper \cite{KKO} constructs a
compactification of maps from
curves to  $X$
without allowing any domain component collapse to points. There, the targets are the fibers of
$X[n]^+$, the universal family over the Fulton-MacPherson configuration spaces.
 Precisely, modify $X$ by
 blowing up points $x$ where the components collapse and then gluing copies
of  $\PP (T_x\oplus \CC )$  along
the exceptional divisors $\PP(T_x)$ to obtain a new target.  For the relative version of \cite{KKO} with respect to $D$,
it is natural to use  the
fibers of $X_{D}[n]^+$ as targets.
The statement (1), (2), and (3) of Theorem \ref{Delta} will be some key
ingredients for establishing the properness and the perfect obstruction theory of the moduli space of such maps.

\subsection{Notation}\label{notation}
\begin{itemize}
\item As in \cite{FM}, for
 a subset $I$ of $N:=\{1,2,...,n\}$, let \[ I^+:=I \cup \{ n+1
 \}. \]

\item Let $Y_1$ be the blowup of a nonsingular complex variety $Y_0$
along a nonsingular closed subvariety $Z$. If $V$ is an
irreducible subvariety of $Y_0$, we will use $\widetilde{V}$ or
$V(Y_1)$  to denote

\begin{itemize}
\item the total transform of $V$, if $V \subset Z$;

\item the proper transform of $V$, otherwise.
\end{itemize}

If there is no risk to cause confusion, we will use simply $V$ to
denote $\widetilde{V}$. The space $\mathrm{Bl}_{\widetilde{V}}Y_1$
will be called the iterated blowup of $Y_0$ along centers $Z, V$
(with the order).

\item For a partition $I=\{
I_0, I_1,...,I_l \}$ of $N$, $\Delta _I$ denotes the
polydiagonal associated to $I$. We will also consider the binary operation
$I\wedge J$ on the set of all partitions defined by \[ \Delta
_I\cap \Delta _J = \Delta _{I\wedge J} \] as in \cite{Ul} (page
143).  We use $\Delta _{I_0}$ instead of $\Delta _I$ when $I=\{
I_0, I_1,...,I_l \}$ such that $|I_i|=1$ for all $i\ge 1$.

\item We
say that a collection $\mathcal{C}$ of closed subvarieties in a variety  {\em meets} or {\em
intersects transversely} if, for every pair of two disjoint
nonempty subsets $\mathcal{C}_1$ and $\mathcal{C}_2$ of
$\mathcal{C}$, the two subvarieties $\bigcap \mathcal{C}_1:=\bigcap _{Z\in \mathcal{C}_1}Z$ and
$\bigcap \mathcal{C}_2$ meet transversely (this includes the case
that they are disjoint).

\end{itemize}

\subsection{Acknowledgements} The authors thank Daewoong Cheong, Li Li, Yong-Geun
Oh, and Dafeng Zuo for useful discussions.
We also thank Ionu\c t Ciocan-Fontanine and Referee for valuable comments.
B.K. is partially
supported by NRF grant 2009-0063179.

\section{Proof of Main Theorems}
\subsection{Wonderful Compactifications}

We recall some results in \cite{Li1} which are needed in this
paper.

\medskip

A finite collection $\mathcal{G}$ of nonsingular, proper, nonempty
subvarieties of a nonsingular algebraic variety $Y$ is called a
{\em building set} if the following two conditions are satisfied.

\begin{enumerate}

\item For every $V$ and $W$ in $\mathcal{G}$, they intersect
cleanly, that is, the tangent bundle $T(V\cap W)$ of the
intersection coincides with the intersection of tangent bundles
$TV$ and $TW$ in $TY$.

\item  For the intersection $\bigcap\mathcal{C}$ of a subset
$\mathcal{C}$ of $\mathcal{G}$, an element $V$ in $\mathcal{G}$ is
called a {\em $\mathcal{G}$-factor} of $\bigcap\mathcal{C}$ if
\begin{itemize}
\item $V$ contains $\bigcap\mathcal{C}$ and
\item there is no
other $V'$ in $\mathcal{G}$, contained in $V$ and containing
$\bigcap\mathcal{C}$.

\end{itemize}

\noindent Then the second condition is as follows. The
collection $\mathcal{C}'$ of all $\mathcal{G}$-factors of
$\bigcap\mathcal{C} $ meets transversely and the intersection
$\bigcap\mathcal{C}' $ is exactly $ \bigcap\mathcal{C}$.

\end{enumerate}


Define the so-called {\em wonderful compactification}
$Y_{\mathcal{G}}$ of $Y$ with respect to $\mathcal{G}$ to be the
closure of $Y\setminus \bigcup _{V\in \mathcal{G}} V$ diagonally
embedded in
\[ Y\times \prod _{V\in\mathcal{G}}\mathrm{Bl}_{V}Y .\]
It has the following wonderful properties.

\begin{Thm}\label{Li1}{\em(\cite{Li1})}
\begin{enumerate}

\item The variety $Y_{\mathcal{G}}$ coincides with the iterated
blowup of $Y$ along all $V$ in $\mathcal{G}$ whenever the order of
centers $V$ is an inclusion order, or a building set order.

\item The boundary $Y_{\mathcal{G}}\setminus (Y\setminus \bigcup
_{V\in \mathcal{G}} V)$ is the union of divisors $\widetilde{V}$,
corresponding to $V\in \mathcal{G}$. The divisors
intersect transversally.

\item A subset $\mathcal{C}$ of $\mathcal{G}$ is nested if and
only if the intersection of all divisors $\widetilde{V}$, for
$V\in\mathcal{C}$, is nonempty

\end{enumerate}
\end{Thm}

We explain terminologies used in Theorem \ref{Li1}.
 An {\em inclusion
order} (resp. a {\em building set order}) above is by definition a
total order $V_1,...,V_l$ of $\mathcal{G}=\{V_1,...,V_l \}$ if $i
< j$ whenever $V_i\subset V_j$ (resp. if $V_1,...,V_k$ form a
building set for any $k=1,...,l$). Hence, $Y_{\mathcal{G}}\cong
\mathrm{Bl}_{V_l}...\mathrm{Bl}_{V_1}Y$ as $Y$-varieties. Here one
should recall the convention \ref{notation} on the centers. A
subset $\mathcal{C}$ of a building set $\mathcal{G}$ is called
{\em nested} if there are a positive integer $k$ and a flag
$(W_1\subset W_2 \subset ... \subset W_k)$ such that every element
of $\mathcal{C}$ is a $\mathcal{G}$-factor of some $W_i$. Here
$W_i$ is an intersection of elements of $\mathcal{G}$.

\bigskip

For example, the Fulton-MacPherson configuration space $X[n]$ is
the wonderful compactification of $X^n$ with respect to the
building set $\{\Delta _I \subset X^n \ | \ I\subset N, |I|\ge
2\}$.

\subsection{Proof of Theorem  \ref{D} and Inductive
Construction of $X_{D}^{[n]}$}\label{Dproof}

Note that the collection of all subsets $D_{c,S}$ in $X^n$ is a
building set. Hence parts (\ref{D1}), (\ref{D3}) and
(\ref{D4}) of Theorem \ref{D} follow from Theorem \ref{Li1}. In particular,
$X_D^{[n]}$ can be constructed by iterated blowups of $X^n$ along
nonsingular centers (and the proper transforms of)  $$D_S :=\coprod
_{c,S} D_{c,S}$$ arrayed by an inclusion order. We may reshuffle
centers as:
\[ D_{\{1\}};D_{\{1,2\}},D_{\{2\}};D_{\{1,2,3\}},D_{\{1,3\}},D_{\{2,3\}},D_{\{3\}};...;
D_{\{1,2,...,n\}},...,D_{\{n\}},
\] keeping the same result $X_D^{[n]}$ after the blowup along the centers with
this building set order.

 The above ordering of centers provides an inductive construction of $X_D^{[n]}$.
 Define $X_D^{[n]+}$ to be the iterated blowups of
$X^{[n]}_D\times X$ along centers $D_{T^+}$, arrayed by an inclusion order, where $T^+=T\cup\{ n+1\}$,
$T\subset N$, and $|T|\ge 1$. (This space is not isomorphic to $X_D^{[n+1]}$
unless $D$ is a divisor.) Note that the flatness of the natural
projection $X_D^{[n]+} \ra X_D^{[n]}$ in Theorem \ref{D}
(\ref{D2}) holds since it is a map between nonsingular varieties
with equi-dimensional fibers. The projection is equipped with
sections provided by $\Delta _{\{i\}^+}\subset X_D^{[n]+}$,
$i=1,...,n$.

\subsection{Proof of Theorem \ref{Delta} and Inductive Construction of  $X_{D}[n]$}

We would like to take a sequence of blowups starting from $X^n$
along centers $D_S$ and $\Delta _I$, $S, I\subset N$, $|S|\ge 1,\
|I|\ge 2$. However they do not form a building set. (See Remark
3.3 for an example.) Hence we cannot apply Theorem \ref{Li1}
directly to $Y=X^n$. Instead, we use the wonderful
compactification in a two-step process. We will show in Proposition
\ref{poly} that altogether the proper transforms
$\widetilde{\Delta}_I$ of $\Delta _I$ in $X_D^{[n]}$ form a
building set. Therefore we can apply Theorem \ref{Li1} to
$Y=X^{[n]}_D$ with the building set $\{ \widetilde{\Delta} _I
\}_{I}$ where $I\subset N, |I|\ge 2$.
The technical lemma on blowups will be deferred to Lemma \ref{blowupLemma} at the
end of this subsection.
\medskip

The inductive construction starting from $X^n$ is given by the
iterated blowup with the order:
\[ \begin{array}{l} D_{\{1\}}; \\ D_{\{1,2\}},D_{\{2\}},\Delta
_{\{1,2\}}; \\
D_{\{1,2,3\}},D_{\{1,3\}},D_{\{2,3\}},D_{\{3\}}, \Delta
_{\{1,2,3\}},\Delta _{\{1,3\}},\Delta
_{\{2,3\}}; \\ \ \ \ \ \ \ \ \ \  \vdots   \\
D_{\{1,2,...,n\}},...,D_{\{1,n\}},...,D_{\{n-1,n\}}, D_{\{n\}},
\Delta _{\{1,2,...,n\}},...,\Delta _{\{1,n\}},...,\Delta
_{\{n-1,n\}}.
\end{array}\]

One can achieve this sequence from the sequence of the building
set orders:
\[\begin{array}{l} D_{\{1\}};D_{\{1,2\}},D_{\{2\}};D_{\{1,2,3\}},D_{\{1,3\}},D_{\{2,3\}},D_{\{3\}};...;
D_{\{1,2,...,n\}},...,D_{\{n\}};\\ \Delta _{\{1,2\}};\Delta
_{\{1,2,3\}},\Delta _{\{1,3\}},\Delta _{\{2,3\}};...; \Delta
_{\{1,2,...,n\}},...,\Delta _{\{n-1,n\}} \end{array} \]

 To see it, first note that all the centers $D_T$ and
$\Delta _I$ are \'etale locally linearized simultaneously in $X^n$,
and hence in an iterated blowup of $X^n$ along any set of the
centers, by Lemma \ref{blowupLemma} (\ref{linear}). In particular
this shows that the divisor $D_T$ is transversal to $\Delta _I$ in
any iterated blowup of $X_D^{[n]}$ along any set of all the
centers. Now we may rearrange the centers from the initial order
using the reordering of two transversal centers (Lemma
\ref{blowupLemma} (\ref{transversal})).

\medskip

 Define $X_D[n]^+$ as the blowup of $X_D[n]\times X$ along
$D_{S^+}$, $\Delta _{I^+}$, more precisely, along $D_{S^+}$ with the inclusion order first, then along
$\Delta _{I^+}$, also with the inclusion order, where
$S,I\subset N$ and $|S|\ge 1$, $ |I|\ge 2$. As before, the
projection $X_D[n]^+\ra X_D[n]$ has the sections provided by
$\Delta _{\{i\}^+}\subset X_D[n]^+$, $i=1,...,n$.

\begin{Prop}\label{poly}
\begin{enumerate}

\item\label{building1} Let $I_1$ and $I_2$ be partitions of $N$.
Then the intersection of proper transforms $\widetilde{\Delta}
_{I_1}$ and $ \widetilde{\Delta} _{I_2}$ in $X_D^{[n]}$ is the
proper transform $\widetilde{\Delta} _{I_1\wedge I_2}$ of the
intersection $\Delta _{I_1}\cap \Delta _{I_2}=\Delta _{I_1\wedge
I_2}$.

\item\label{building2} The collection of all diagonals
$\widetilde{\Delta} _I$, $I\subset N$, $|I|\ge 2$, is a building
set in $X_D^{[n]}$.

\end{enumerate}
\end{Prop}

\begin{proof}
Note that $\Delta _I$ in $X_D^{[n]}$ coincides with the variety
defined by equations \[ \sigma  _a = \sigma _b, \ \forall a,b\in
I_i,\ I_i\in I
\] where
$\sigma _a$ is the section of $X_D^{[n]+}\ra X_D^{[n]} $, induced
by $\Delta _{\{a\}^+}$. This can be seen by considering the
imposed equation at {\em general} points. Now the proof is
straightforward.
\end{proof}

 \noindent{\em Proof of Theorem \ref{Delta}
(\ref{nested}).} For simplicity assume that $D$ is connected.

($\Rightarrow$). The condition on the pair $S_i$ and $S_k$ ($I_j$
and $I_l$, respectively) is a direct consequence of Theorem
\ref{Li1}. Suppose that both $S \cap I$ and $I\setminus S$ are
nonempty. Then Lemma \ref{blowupLemma} (\ref{disjoint}) shows that
$\widetilde{D}_S\cap \widetilde{\Delta} _I$ is empty.

($\Leftarrow$). Let  $\{ D_{S_i}, \Delta _{I_j} \}_{i,j}$ be a
nested set and let $V$  be the transversal intersection $\bigcap
_i D_{S_i} (X^{[n]}_D)$. Then an argument similar to the proof of
Proposition \ref{poly} shows that the collection
\[ \mathcal{G}:= \{V\cap \Delta _{I} (X^{[n]}_D) \ | \ I\subset
N,\ |I|\ge 2, \ \{ S_i, I\} _i \text{ is nested} \}
\] is a building set of $V$.
According to Lemma \ref{blowupLemma} (\ref{funct}),
$\widetilde{V}$ in $X_D[n]$ coincides with the wonderful
compactification $V_{\mathcal{G}}$ of $V$. Now since $\{ V\cap
\Delta _{I_j} (X^{[n]}_D)\}_j$ is nested, we conclude that
$\widetilde{V}\cap \bigcap _j\widetilde{\Delta} _{I_j}$ in
$X_D[n]$ is nonempty and transversal by Theorem \ref{Li1}. Also,
$\tilde{V}$ is $\bigcap D_{S_i}(X_D[n])$ due to Lemma
\ref{blowupLemma} (\ref{iso1}) and $\widetilde{\bigcap D _{S_i}}
\subset \bigcap \widetilde{D _{S_i}}$ in $X_D[n]$. This completes
the proof. \hfill$\Box$

\medskip

Note that the above proof of (\ref{nested}) shows the statement
(3) of Theorem 2 is also true.

\begin{Lemma}\label{blowupLemma}
Let $Z, Z_i, V, V_i$, $i=1,...,k$  be nonsingular subvarieties of
a nonsingular variety $X$, let $\pi:  \mathrm{Bl}_Z X \rightarrow
X$ be the blowup map along $Z$ and let $E$ be the exceptional
divisor.

\begin{enumerate}

\item\label{transversal} If $Z_1$ and $Z_2$ intersect transversely, then
$\mathrm{Bl}_{\widetilde{Z_2}}\mathrm{Bl}_{\mathrm{Z_1}} Y =
\mathrm{Bl}_{\widetilde{Z_1}}\mathrm{Bl}_{\mathrm{Z_2}} Y $.

\item\label{linear} If $Z$, $V_i$, $i=1,...,k$ are \'etale locally linearized in
$X$ simultaneously, then so are their transforms in
$\mathrm{Bl}_ZX$, and in particular $V_i$ and $V_j$ for any $i,j$
intersect cleanly.

\item\label{iso1} If $V$ meets $Z$ transversally, then
$\widetilde{V} = \pi^{-1}(V)$.

\item\label{funct} If $V$ and $Z$ intersect cleanly and $V$ is not
contained in $Z$, then $\widetilde{V}$ is the blowup of $V$ along
$Z\cap V$.

\item\label{disjoint} Assume that  $V_1$ and $ V_2$ intersect
cleanly.
 If $V_1\cap V_2\subset Z  \varsubsetneq V_1$, then $\widetilde{V}_1$and
$\widetilde{V}_2$ are disjoint.

\end{enumerate}
\end{Lemma}

\begin{proof} The only nonstandard result is
(\ref{disjoint}), which we prove here. Assume that they are not
disjoint. Then for some point $p\in V_1\cap V_2$, there are $v_i
\in T_pV_i$ such that in the normal bundle $N_{Z/X}$, $
[v_1]=[v_2]\ne 0$. Since $TZ \subset TV_1$, $v_2$ is an element of
$TV_1$ as well as $TV_2$. It implies that $[v_2] = 0$ in $N_{Z/X}$
since $TV_1 \cap TV_2 =T(V_1 \cap V_2) \subset TZ$. This is a
contradiction.
\end{proof}

\section{Some more properties}

\subsection{Stable degenerations}\label{expuniv} For simplicity
assume that $D$ is connected. Note that $X_D^{[n]+}\ra X_D^{[n]}$
is a flat family of stable degenerations of $X$ with $n$ smooth
labeled points away from $D$ (see subsection 2.2). The labeled points may not be
distinct. Stability means that  every closed fiber $F$ has no
nontrivial automorphism fixing the following data:
the natural map $F \ra X$; $F\cap \widetilde{D}_{\{n+1\}}$; and
the marked points $F\cap \widetilde{\Delta} _{\{i,n+1\}}$,
$i=1,...,n$. The fibers are normal crossing varieties, \'etale
locally the form  $xy=0$. The generic fiber over $D_S(X_D^{[n]})$ is
the coproduct
$$\mathrm{Bl}_D X\coprod _{\PP (N_{D/X})} \PP (N_{D/X}\oplus {\bf 1})$$
of  $\mathrm{Bl}_D X$ and $\PP (N_{D/X}\oplus {\bf 1})$ along $\PP
(N_{D/X})$. The points labeled by $a \in S$ are in $\PP
(N_{D/X}\oplus {\bf 1})\setminus (\PP (N_{D/X}) \cup \PP ({\bf
1}))$ and the other points are in $\mathrm{Bl}_D X \setminus \PP
(N_{X/S})$. In general, $\Delta _{\{a\} ^+}$ is disjoint from
$D_{\{n+1\}}$ in $X_D^{[n]+}$ by Lemma \ref{blowupLemma}
(\ref{disjoint}).

\medskip

Similarly, $X_D[n]^+ \ra X_D[n]$ is a flat family of stable
degenerations of $X$ with $n$ {\em distinct} smooth labeled points
away from $D$ (see subsection 2.3). It is equipped with sections
$\sigma _i$, which are disjoint to each other. Specifically, the
fibers of $X_D[n]^{+}$ over points in the boundary of $X_D[n]$ are
Fulton-MacPherson stable degenerations of fibers of $X_D^{[n]+}$:
In a fiber $F$ of $X_D^{[n]+}$ the labeled points $F\cap
\widetilde{\Delta}_{\{i,n+1\}}$, $i=1,...,n$, are away from
$\widetilde{D}:=F\cap \widetilde{D}_{\{n+1\}}$, but may come
together at some points of $F\setminus \widetilde{D}$. Blow up all
such points $x\in F\setminus \widetilde{D}$ and then glue copies
of $\PP (T_x\oplus \CC )$  along the exceptional divisors
$\PP(T_x)$ to obtain a new modification of $X$ in which the points
in the configuration are now distinct. The stability is similar to
the above case.

\subsection{Group Action by $S_n$} Let $S_n$ be the symmetric
group on $n$ letters. There is a natural $S_n$ action on the space
$X_D[n]$ such that the projection $X_D[n]\ra X^n$ is
$S_n$-equivariant. By Theorem 5.2 in \cite{BG}, all stabilizers
are solvable.

\subsection{Remark}
In general, the space $X_D[n]$ is not isomorphic to the one-step
closure of $(X\setminus D)^n\setminus \bigcup _{|I|\ge 2} \Delta
_I$, that is, the closure in the product
\[ X^n\times  \prod _{c,S\subset N}\mathrm{Bl}_{D_{c,S}}X^n\times
\prod _{I\subset N,\  |I|\ge 2} \mathrm{Bl}_{\Delta _I} X^n . \]
For example, take $X=\CC ^2$ with $D=\{ (x,y)\in\CC ^2 \ | \ y
=0\}$ and consider the limits of $((t,at),(2t,bt))$, as $t$ goes
to $0$. Then the limit in $X_D[2]$ does not depend on $a,b$.
However the limit in the one-step closure depends on $a$, $b$.

\subsection{Examples}

\subsubsection{$\overline{M}_{0,n}$} Let $n\ge 3$. The moduli space
$\overline{M}_{0,n}$ of $n$-pointed  stable rational curves
coincides with $X_{D}[n-3]$ where $X=\PP ^1$ and $D$ consists of three
distinct  points. Indeed, the inductive construction is
exactly the blowup construction of $\overline{M}_{0,n}$ given by
Keel (\cite{Keel}).

\subsubsection{$T_{d,n}$} Let $n\ge 2$. Take $X=\PP ^d$ and let $D$
be a hyperplane. Note that the group $G$ of automorphism of $X$
fixing all points in $D$ is isomorphic to $\CC ^*\ltimes \CC ^d$.
The natural action of the group $G$  on $X_D[n]$ is free and the
quotient $X_D[n]/G$ is isomorphic to the compactification
$T_{d,n}$ studied by Chen, Gibney, and Krashen \cite{CGK}. It
compactifies the configuration space of $n$ distinct labeled
points in $\CC ^d$ modulo $\CC ^*\ltimes \CC ^d$.

\bibliography{math}


\def\cprime{$'$}
\providecommand{\bysame}{\leavevmode\hbox
to3em{\hrulefill}\thinspace}

\end{document}